\documentclass[12pt]{amsart}
\usepackage{amssymb}
\usepackage[all]{xy}
\usepackage{fullpage}
\allowdisplaybreaks

\numberwithin{equation}{section}

\theoremstyle{plain}
\newtheorem{theorem}[equation]{Theorem}
\newtheorem{lemma}[equation]{Lemma}
\newtheorem{prop}[equation]{Proposition}
\newtheorem{cor}[equation]{Corollary}

\theoremstyle{definition}
\newtheorem{defn}[equation]{Definition}

\theoremstyle{remark} 
\newtheorem{remark}[equation]{Remark} 
\newtheorem{rk}[equation]{Remark}

\newcommand{\bF}{\mathbb{F}}
\newcommand{\fA}{\mathfrak{A}}
\newcommand{\fM}{\mathfrak{M}}

\newcommand{\Aut}{\mathsf{Aut}}
\newcommand{\Ext}{\mathsf{Ext}}
\renewcommand{\ge}{\geqslant}
\newcommand{\Hom}{\mathsf{Hom}}

\renewcommand{\le}{\leqslant}
\newcommand{\Mat}{\mathsf{Mat}}
\newcommand{\CO}{\mathcal{O}}
\newcommand{\ua}{\uparrow}
\newcommand{\da}{\downarrow}
\newcommand{\bZ}{\mathbb Z}

\author{David Benson, Radha Kessar, and Markus Linckelmann}

\address{David Benson \\ 
Institute of Mathematics\\ 
Fraser Noble Building\\
University of Aberdeen\\ 
King's College\\ 
Aberdeen AB24 3UE\\ 
United Kingdom}

\address{Radha Kessar and Markus Linckelmann \\
School of Mathematics, Computer Science \& Engineering \\
Department of Mathematics \\
City, University of London \\
Northampton Square \\
London EC1V 0HB \\
United Kingdom}



\title{Blocks with normal abelian defect and abelian $p'$ inertial quotient}

\begin{document}

\begin{abstract}
Let $k$ be an algebraically closed field of characteristic $p$,
and let $\CO$ be either $k$ or its ring of Witt vectors $W(k)$.
Let $G$ a finite group and $B$ a block of $\CO G$ with normal
abelian defect group and abelian $p'$ inertial quotient.
We show that $B$ is isomorphic to its second Frobenius
twist. This is motivated by the fact that bounding Frobenius
numbers is one of the key steps towards Donovan's conjecture.
For $\CO=k$, we give an explicit description of the basic algebra
of $B$ as a quiver with relations. It is a quantised version
of the group algebra of the semidirect product $P\rtimes L$.

\end{abstract}

\maketitle

%%%%%%%%%%%%%%%%%%%%%%%%%%%%%%%%%%%%
\section{Introduction}

The purpose of this paper is to bound the Frobenius numbers and to
give a structure theorem for blocks of finite groups with normal 
abelian defect groups and abelian $p'$ inertial quotients. This 
extends the results of Benson and Green~\cite{Benson/Green:2004a}, 
Holloway and Kessar~\cite{Holloway/Kessar:2005a}, 
Benson and Kessar~\cite{Benson/Kessar:2007a}.

We show that these blocks are isomorphic to their second Frobenius 
twist. By \cite{Kessar:2004}, bounding Frobenius numbers is a key step
towards Donovan's conjecture; see for instance \cite{EatonLivesey:2018a},
\cite{EatonLivesey:2018b}. We obtain further a complete description of 
the basic algebra of such a block over a field by means of quiver with 
relations.  

Our main theorems are as as follows. 
Let $k$ be an algebraically closed field of characteristic $p$ and let 
$W(k)$ be the ring of Witt vectors over $k$. Let $\CO\in$ $\{k, W(k)\}$. 
For $q$ a power of $p$, the Frobenius automorphism $\lambda\mapsto\lambda^q$
of the field $k$ lifts uniquely to an automorphism of the ring $W(k)$, and
we denote its inverse in both cases by $\mu\mapsto$ $\mu^{\frac{1}{q}}$.

Recall from ~\cite{Benson/Kessar:2007a} that for an $\CO$-algebra 
$A$, the Frobenius twist $A^{(q)}$ is the $\CO$-algebra which equals $A$ 
as a ring, and where scalar multiplication is twisted  via the  
Frobenius map; that is, for $\lambda \in \CO$, and  $a\in A$, the  
action on  $A^{(q)}$ is given by $\lambda  \cdot a =$
$\lambda ^{\frac{1}{q}} a$.

\begin{theorem}  \label{th:Fr}  
Let $P$ be a finite abelian $p$-group, $L$ an abelian $p'$-subgroup of   
$\Aut(P)$ and  $\alpha \in H^2(L, \CO^{\times}) $.  The twisted group   
algebra $\CO_{\alpha} (P \rtimes L)$ is  isomorphic to its second 
Frobenius twist  $\CO_{\alpha} (P \rtimes L)^{(p^2)}$.  
\end{theorem} 

Theorem~\ref{th:Fr}  is proved  in Section 2.

\begin{theorem}  \label{th:qu}
Let $P$ be a finite abelian $p$-group, $L$ an abelian $p'$-subgroup of 
$\Aut(P)$ and $\alpha \in H^2(L, k^{\times})$ and $\fA$ the basic 
algebra of the twisted group algebra $k_{\alpha}(P \rtimes L)$. Then  
$k_{\alpha}(P \rtimes L)$ is a matrix algebra over $\fA$ and $\fA$ has 
an explicit presentation as a quantised version of the  group algebra 
of the semidirect product $P\rtimes L$.
\end{theorem}

The explicit generators and relations for $\fA$ are given in 
Theorem~\ref{th:A}.

By a theorem of K\"ulshammer~\cite{Kulshammer:1985a}, any block of a  
finite group algebra over $\CO$ with a normal defect group is isomorphic 
to a matrix algebra of a twisted  group algebra of the semi-direct 
product of  the defect group of the block with the inertial quotient 
of the block.  Combining  \cite{Kulshammer:1985a} with the two results 
above  yields the following.
  
\begin{cor} \label{cor:Fr} 
Let $B$ be a block of a  finite  group algebra over $\CO$  with a 
normal defect group $P$ and abelian inertial quotient $L$. Then $B$  
is isomorphic to its second Frobenius twist and if $\CO=k$, then
$B$ is a matrix algebra over a quantised version of the  group algebra 
of the semidirect product $P\rtimes L$.   
\end{cor}

\begin{remark}
It seems unclear whether the same bound holds for strong Frobenius
numbers, introduced by Eaton and Livesey in \cite{EatonLivesey:2018b}.
One issue is that we do not have a sufficiently explicit 
description of the automorphism $\varphi$ of $P\rtimes L$ constructed 
in Lemma \ref{lem:frobaut}  below.
\end{remark}

\bigskip

\noindent
\textbf{Acknowledgement.} This material is based on work supported
by the National Science Foundation under Grant No.\ DMS-1440140 while 
the authors were in residence at the Mathematical Sciences Research Institute
in Berkeley, California, during the Spring 2018 semester.  
The first author thanks City, University of London for its hospitality 
during part of the preparation of this paper.
The third author acknowledges  support from EPSRC grant EP/M02525X/1.

%%%%%%%%%%%%%%%%%%%%%%%%%%%%%%%%%%%%%%%%%%%%%%%%%
\section{Proof of Theorem~\ref{th:Fr}.}

For $L$ a group, $\phi \in \Aut (L)$ and $\alpha: L \times L \to$
$\CO^{\times}$, denote by  $\, ^\phi  \alpha : L\times L \to \CO^{\times} $  
the map defined by $\,^{\phi} \alpha (x, y) =$
$\alpha (\phi^{-1} (x), \phi^{-1}(y) )$  and  by $(\phi,  \alpha ) \to$
$\,^{\phi }\alpha$ the induced action of $\Aut(L)$ on 
$H^2(L, \CO^{\times})$. For $q$ a power of $p$, denote by $\alpha ^{(q)}$ 
the  map $ L \times L \to \CO^{\times}$ defined by 
$\alpha ^{(q)}( x, y) =  \alpha (x, y) ^{\frac{1}{q} }  $  
and by $\alpha ^{(q)}$ the  image of $\alpha$ under the induced  
isomorphism  $H^2(L, \CO^{\times}) \cong  H^2(L, \CO^{\times})   $.   

\begin{lemma} \label{lem:autfrob}  
Let $L$ be a finite abelian $p'$-group and let $\phi: L \to L $ be the   
group automorphism defined by $\phi(x)= x^p$ for all $x\in$ $L$. 
Then for all $\alpha \in H^2(L, \CO^{\times} ) $, we have 
$ \, ^{\phi} \alpha=  \alpha ^{(p^2)} $. 
\end{lemma} 

\begin{proof} 
It is well-known that since $L$ is a finite $p'$-group, it follows that 
the canonical map $\CO\to$ $k$ induces an isomorphism 
$H^2(L,\CO^\times)\cong$ $H^2(L,k^\times)$. Thus we may assume
that $\CO=k$.

Consider the universal coefficient sequence
\[ 0 \to \Ext^1(H_1(L,\bZ),k^\times) \to H^2(L,k^\times) \to
  \Hom(H_2(L,\bZ),k^\times) \to 0. \]
Since $k$ is algebraically closed, $k^\times$ is divisible,
and therefore injective as an abelian group. So the first
term in this sequence is zero. For the third term, we have
$H_2(L,\bZ)\cong \Lambda^2(L)$, the exterior square in
the category of abelian groups. Therefore we obtain
an isomorphism
\[ H^2(L, k^\times) \cong \Hom(\Lambda^2(L),k^\times). \]
which   by naturality is $\Aut (L)$ equivariant  and which commutes 
with  the  Frobenius morphism of  $k$. More precisely, if     
$\alpha \in H^2(L,k^\times)$ corresponds to  $\tau \in   
\Hom(\Lambda^2(L),k^\times)$  under the above isomorphism, then  for 
any $\psi \in \Aut(L)$,  and any power $q$ of $p$, $\,^{\psi} \alpha$ 
corresponds to  the homomorphism  $\,^{\psi}  \tau$   defined by 
$\,^{\psi} \tau (x\wedge y) = 
\tau (\psi^{-1} (x) \wedge\psi^{-1} (y)) $ and $\alpha ^{(q)} $ 
corresponds to the homomorphism $\tau^{(q)}$  defined by 
$ \tau^{(q)}(x\wedge y) = \tau (x\wedge y)^{\frac{1}{q}} $. The result 
follows since for any $\tau \in \Hom(\Lambda^2(L),k^\times)$ we have 
$\tau(x^q \wedge y) = \tau(x\wedge y)^q = \tau (x\wedge y^q)$. 
\end{proof}

The isomorphism $H^2(L, k^\times) \cong \Hom(\Lambda^2(L),k^\times)$
in the above proof can be explicitly described as follows. 
If $\alpha\in$ $Z^2(L,k^\times)$, then the image of the class of
$\alpha$ in $\Hom(\Lambda^2(L),k^\times)$ is the group homomorphism
$x \wedge y \mapsto$ $\alpha(x,y)\alpha(y, x)^{-1}$, where $x$, $y\in$ 
$L$. One can either verify directly, using the $2$-cocycle identity, 
that this assignment is a group homomorphism in each component, or one 
can observe that $\alpha(x,y)\alpha(y,x)^{-1}$ is equal to the commutator 
of lifts of $x$, $y$ in a central extension determined by $\alpha$.
More precisely, let
\[ 1\to k^\times \to \tilde L \to L \to 1 \]
be a central extension defined by $\alpha$. For each $x \in L$, choose  
an  element $\tilde x \in \tilde L$ lifting $x$. An easy calculation
shows that  $\alpha(x,y)\alpha(y,x)^{-1}=$ $[\tilde x, \tilde y]$. 
This commutator does not depend on the choices of the lifts $\tilde x$ 
and since $\tilde L$ is a central extension of the abelian group $L$, 
this commutator is a group homomorphism in each component. In 
particular, $[\bar{x^p}, \bar{y^p}]=$ $[\bar x, \bar y]^{p^2}$, which 
explains the statement of the above Lemma.  

\begin{lemma} \label{lem:pprimeauto}  
Let $ P$ be a finite $p$-group and let  $\Phi(P)$ be the Frattini 
subgroup of $P$. The kernel  of the natural group homomorphism  
$ \Aut(P)  \to  \Aut(P/\Phi(P))  $ is a $ p$-group.  If $P$ is  
homocyclic, then the  map  $ \Aut(P) \to \Aut(P/\Phi(P))$ is surjective.
\end{lemma} 

\begin{proof}  
For the first assertion 
see~\cite[Chapter 5, Theorem 1.4]{Gorenstein:1968a}.  Assume that $P$ 
is homocyclic  and let $ \{x_1,  x_2, \cdots,  x_r  \} $ be a  minimal 
generating set of $P$.  Let $ \psi \in  \Aut ( P/ \Phi(P)) $.   For 
each $ i$, $ 1\leq i \leq r$, pick an element  $u_i \in P $  such that 
$\psi ( x_i \Phi(P))  =  u_i \Phi( P) $.  Since    $P$ is homocyclic,  
there exists  a  homomorphism  $\tilde  \psi : P \to P $ such that  
$ \tilde \psi  (x_i)  = u_i $, $ 1\leq i \leq r $. Clearly,  
$\tilde \psi $ lifts $\psi $. Since $\psi$ is an automorphism,   
$ \mathrm{Im}(\tilde \psi) \Phi(P)  =  P $ whence 
$\mathrm{Im}(\tilde \psi) =   P$ and $\tilde \psi \in \Aut(P) $. 
\end{proof}  

\begin{lemma}  \label{lem:frobaut}  
Let $ P$ be a  finite  abelian $p$-group and let $L$ be an abelian   
$p'$-group  acting on  $P$. There exists an automorphism $\phi$ of 
$P\rtimes L$ such that $\phi(L)= L$ and $\phi(x) = x^p$ for all 
$x\in L$.  
\end{lemma}

\begin{proof}   
Denote  by  $L'$ the image of  $L$  in $\Aut (P) $. Proving the  
existence of   $\phi$  is equivalent  to showing that there exists  
$\tau  \in \Aut(P) $ of $P$   such that    $\tau  y \tau^{-1} = y^p $     
for all $ y \in L'$.    Suppose first that  $P  = P_1 \times P_2 $  for 
$L'$-invariant subgroups   $P_1$ and $ P_2$  of $P$.  If  for each 
$i=1,2 $, there exists    $\tau_i \in \Aut(P_i) $   with   
$\tau_i  ({y {\da_{P_i} }})\tau_i^{-1} =     ({y {\da_{P_i} }})^p $ for 
all $ y\in L'$, then the map $\tau \colon P \to P  $ sending $x_1x_2  $ 
to  $\tau _1(x_1) \tau_2 (x_2) $  for $x_1 \in P_1 $,   $x_2 \in P_2$  
has the required properties. Hence we may assume  that  $P$  is 
indecomposable for the action of $L'$ and consequently that  $P$ is 
homocyclic (see  \cite [Chapter 5, Theorem 2.2]{Gorenstein:1968a}).   

We claim that  it suffices to prove the result  for  the case that $P$ 
is elementary abelian.  Indeed, let   $U$ be the kernel of the  map 
$\Aut(P) \to  \Aut(P/\Phi(P))  $. By Lemma~\ref{lem:pprimeauto},  $U$ is 
a $p$-group.    Let  $\bar L' $ be  the  image of $L' $ in   
$ \Aut (P/\Phi(P)) $ and suppose   that  there exists    
$\bar \tau \in \Aut (P/\Phi (P) )  $ such that   
$\bar \tau \eta \bar \tau^{-1} =  \eta^{p}$ for all  $ \eta \in \bar L $.  
By  Lemma~\ref{lem:pprimeauto}, there exists $\tau \in \Aut (P) $  
lifting $\bar \tau $. Since $L'U$ is the full inverse image of $\bar L'$ 
in $\Aut(P)$, and  $\bar L' $ is  $\bar  \tau $-invariant,   
$\tau  L' U  \tau ^{-1}   = L'U  $. Hence   $ L'$ and 
$ \tau L' \tau^{-1} $  are both  complements  to   the normal  Sylow 
$p$-subgroup  $U$  of $L'U $.   By the Schur--Zassenhaus theorem, there 
exists  $u \in U $ such that   $uL'u^{-1}= \tau   L'  \tau^{-1} $. 
Replacing $\tau $  by $ u^{-1} \tau $ we may assume that  
$\tau L'  \tau^{-1}= L'$. Then for any $y \in L '$, $\tau  y' \tau^{-1}$ 
and   $ y'{^p} $ are   elements  of   the $p'$-group  $L'$  lifting   the 
same element of $\bar L' $.
The claim  follows by Lemma~\ref{lem:pprimeauto}.
 
By the  discussion above we may assume  that  $P$ is elementary abelian  
and that $P$ is an indecomposable, faithful ${\mathbb F}_p L' $-module.  
Since  $L'$ is an abelian  $p'$-group, $ P$  is  in fact an irreducible  
${\mathbb F}_p L' $-module and  $L'$ is cyclic.  Let 
$L' =\langle y  \rangle  $  and let $f (X)  \in {\mathbb F}_p[X]$ be the 
characteristic polynomial of $ y$  as  an element of 
$\mathrm{End}_{{\mathbb F}_p}(P) $. Since $f(y^p)  = f(y) ^p=0$,  $f(X)$ 
is also the characteristic polynomial of $y^p$. Thus, $y$ and $y^p$ are 
conjugate in $\mathrm{GL} (P)\cong \Aut (P)$.
 \end{proof}

\begin{proof}[{Proof of Theorem~\ref{th:Fr}}] 
Let $\phi $ be as in  Lemma~\ref{lem:frobaut}. Then $\phi$ induces a  
$\CO$-algebra  isomorphism  $\CO_{\alpha} (P\rtimes L) \cong 
\CO_{\,^{\phi} \alpha}(P\rtimes L)$. The result follows by  
Lemma~\ref{lem:autfrob} since  for any power $q$ of $p$, 
$\CO_{\alpha} (P \rtimes L) ^ {(q)}  \cong  
\CO_{\alpha^{(q)}}(P \rtimes L) $ as $\CO$-algebras.
\end{proof} 

%%%%%%%%%%%%%%%%%%%%%%%%%%%%%%%%%%%%%%%%%%%%%%%%%%%%%%%%%%%%%%%
\section{On characters of groups of class two}\label{se:class2}

For a finite group $H$ denote by $\mathrm{Irr} (H)$ the set of ordinary 
irreducible characters of $H$. If $N$ is a normal subgroup of $H$ and 
$\chi \in \mathrm{Irr}(H)$, denote by $\mathrm{Irr}(H \, | \, \chi)$ 
the subset of $\mathrm{Irr}(H)$ covering $\chi$. Recall that  if 
$H/N$ is abelian, then the group of irreducible (i.e. linear) characters 
of $H/N$ acts on $ \mathrm{Irr}(H \, | \, \chi)$ via multiplication and 
this action is transitive.

\begin{prop} \label{pro:class2} 
Let $H$ be a finite group which is nilpotent of class $2$.  Let $\chi$ 
be a faithful irreducible character of $Z:=[H, H]$.  Set   
$m = \sqrt{ |H : Z(H)|} $.   
\begin{enumerate}  
\item    
For any  $ \phi \in \mathrm{Irr} (Z(H) \, | \, \chi)$, 
$\phi{\ua^H} = m \tau_{ \phi} $  for  some  
$\tau_ {\phi}   \in \mathrm{Irr} (H) $.  In particular, 
$m=\tau_\phi (1) $ is an integer.  

\item 
The map $\phi \to \tau _{\phi} $,  
$ \phi \in  \mathrm{Irr} (Z(H) \, | \, \chi)$, is  a bijection between     
$\mathrm{Irr} (Z(H) \, | \, \chi)$ and $\mathrm{Irr} (H \, | \, \chi) $.

\item     
The actions of $\mathrm{Irr}( H/Z)$ on $\mathrm{Irr}(H \, | \, \chi)$ 
and of $\mathrm{Irr}(Z(H)/Z)$ on $\mathrm{Irr}( Z(H) \, | \, \chi)$ are 
compatible  with the bijection in {\rm (ii)}. More precisely, let  
$\eta \in \mathrm{Irr}( H/Z)$, and let $\phi \in 
\mathrm{Irr}(Z(H) \, | \, \chi) $. Then 
$\tau_{ \eta{\da_{Z(H)}} \phi} =\eta \tau_{\phi}$. Consequently, 
$ \eta \tau_{\phi} = \tau_{\phi} $ if and only if $\eta $ restricts to 
the trivial character of $Z(H)$.

\end{enumerate} 
\end{prop}  

\begin{proof} 
Let  $\tau $ be an irreducible character of $H$ covering $\chi $. We 
claim that $\tau (x) =0 $ for all $ x \in H\setminus  Z(H) $. Indeed,   
since  $Z(H) $ is the intersection of  all  maximal abelian subgroups of 
$H$,  it suffices to prove that   if $A$ is a maximal subgroup of $H$, 
then $\tau (x) =0 $  if  $ x \notin A$. So, let $A$ be a maximal abelian 
subgroup of $H$.  Then $ Z(H) \leq  A $ and since $H $  is  of class $2$, 
$ A$ is normal in $H$. Let $\psi $  be an irreducible constituent of the 
restriction of $\tau$ to $A$  and suppose that   $g\in H$  is such that  
${^g}\psi=\psi$. Then for all $a\in A$, $\psi(gag^{-1})=\psi(a)$ and so 
$\psi(gag^{-1}a^{-1})=1$. Since  the restriction of $\psi $ to $ Z $ 
equals $\chi $, we have that $\chi(gag^{-1}a^{-1})=1$  for all 
$ a \in A $.  The faithfulness of $\chi $  and the maximality of $A $  
now  imply  that $  g \in C_G(A) = A $.   Consequently,   $\tau =   
\psi{\ua^H}$ and $\tau (x) =0$ for all $x \notin A$, proving the claim.

Let $ \phi  $ be   the unique linear character of  $Z(H)$ covering 
$\chi$ and which is covered by $\tau$.  Since $\phi$ is linear, the 
restriction of $\tau$ to $Z(H)$ consists of $\tau(1)$ copies of 
$\phi $. By the claim above,
\[  1 =  \langle  \tau, \tau \rangle   =\frac{1}{|H| } \sum_{x\in Z(H) } 
\tau (x) \tau (x^{-1} ) =  \frac{\tau (1)^2}{ |H|} \sum_{x\in Z(H) }  
\phi(x) \phi(x^{-1} ) = \frac{  \tau (1)^2  } {m^2}.  \] 
Thus   $ \tau(1) =m$ and
\[  \tau (1) m  =      |H:Z(H)|=  \phi{\ua^H}(1).\]
On the other hand, by Frobenius reciprocity the  multiplicity of  $\tau $ 
as a constituent of $\phi{\ua^H} $ equals  $\tau (1)  $.     So    
$ \phi{\ua^H}  =     m \tau  $. Setting $\tau_{\phi}= \tau $ proves   
part (i)  of the proposition.  Part (ii)  is immediate from (i)   and 
the fact that  every  element of $\mathrm{Irr} (H \, | \, \chi )  $ 
covers a unique element of $\mathrm{Irr} ( Z(H) \, | \, \chi ) $.   
By the induction formula, $\eta (\phi{\ua^H})  =   
( \eta{\da_{Z(H)}}\phi){\ua^H}$,   hence (i)  gives that     
$\tau_{ \eta{\da_{Z(H)}}  \phi}    =\eta  \tau_{\phi}$.   Now (ii)  
yields that  $\eta \tau _{\phi} = \tau_{\phi} $ if and only if  
$ \eta{\da_{Z(H)}} \phi  =  \phi $  if and only if $ \eta{\da_{Z(H)}}$ 
is trivial.
\end{proof}

%%%%%%%%%%%%%%%%%%%%%%%%%%%%%%
\section{The basic algebra.}

\begin{lemma} \label{le:coh} 
Let $ 1\to  A \to B \to  C  \to 1  $ be a  short exact sequence of 
abelian groups  and let   $ \pi:  D \to    A$ be a surjective 
homomorphism of abelian groups. For each $\alpha  \in C$, choose  
a pre-image  $u_{\alpha } $  in $B$.  Then  there exists  
a  $2$-cocycle  $(\alpha, \beta) \to f_{\alpha, \beta} $   from    
$C \times C$ to $D$ such that $\pi(f_{\alpha, \beta}) =
u_{\alpha}^{-1}u_{\beta}^{-1} u_{\alpha\beta} $   and 
$ f_{\alpha, \beta} =  f_{\beta, \alpha}$ for all $\alpha, \beta \in  
C$.  
\end{lemma}

\begin{proof} 
It is well-known that $\Ext^n_\bZ(C,A)=\{0\}$ for $n\geq$ $2$, and hence 
the connecting homomorphism $\Ext^1_\bZ(C,A)\to$ 
$\Ext^2_\bZ(C,\ker(\pi))=$ $\{0\}$ is zero. Thus the map 
$\Ext_\bZ^1(C,D)\to$ $\Ext_\bZ^1(C,A)$ induced by $\pi$ is surjective. 
In particular, the element in $\Ext^1_\bZ(C,A)$ represented by the given 
short exact sequence lifts to an element in $\Ext_\bZ^1(C,D)$. Rephrased 
in terms of extensions this means that there is a commutative diagram of 
abelian groups with exact rows of the form
$$\xymatrix{1\ar[r] & D\ar[r] \ar[d]_{\pi} & \hat B \ar[r]\ar[d]^{\tau} 
& C \ar[r] \ar@{=}[d] & 1 \\
1 \ar[r] &A\ar[r] & B\ar[r] & C \ar[r] & 1} $$
Note that $\tau$ is surjective and restricts to the map $\pi$ on $D$.
For $\alpha\in$ $C$, choose a preimage $v_\alpha$ of $u_\alpha$ in
$\hat B$ and set $f_{\alpha,\beta}=$ 
$v_\alpha^{-1}v_\beta^{-1}v_{\alpha\beta}$ for all $\alpha$, $\beta\in$ 
$C$. Since $\hat B$ is abelian, it follows that $f_{\alpha,\beta}=$ 
$f_{\beta,\alpha}$. Thus $(\alpha,\beta)\mapsto$ $f_{\alpha,\beta}$ is a 
$2$-cocycle with the properties as stated.
\end{proof}

We recall the following  result on the structure of twisted group algebras.

\begin{lemma}\label{le:twistcentral}  
Let $G$ be a finite group and let  $\alpha\in H^2(G ,k^\times)$. 
Then there exists  a central extension
\[ 1 \to Z \to  \tilde G \to   G \to 1 \]
with $Z$ a finite cyclic $p'$-group and a linear character 
$\chi \colon Z \to k^{\times} $ such that  
$k_{\alpha}  G \cong k\tilde G  e $  where
$e= \frac{1}{|Z|} \sum _{z\in Z}  \chi(z^{-1}) z $  is  the idempotent  
of $kZ$  corresponding to $\chi $.  Moreover,   
$Z$ may be chosen to  be contained in   $[\tilde G, \tilde G] $.
\end{lemma}

\begin{proof}   
This is well known, but for completeness we provide a proof. Let $m$ be 
the order of the cohomology class $\alpha$. Since $k$ is algebraically 
closed, $k^\times$ is a divisible group. So we have a short exact sequence
\[ 0 \to \mu_m \to k^\times \to k^\times \to 0. \]
Consider the corresponding maps of universal coefficient sequences
\[ \xymatrix{0 \ar[r] & \Ext^1(H_1(G,\bZ),\mu_m) \ar[r]\ar[d] & 
H^2(G,\mu_m) \ar[r]\ar[d] &
\Hom(H_2(G,\bZ),\mu_m) \ar[r]\ar[d] & 0\\
0 \ar[r] & \Ext^1(H_1(G,\bZ),k^\times) \ar[r]\ar[d] 
& H^2(G,k^\times) \ar[r]^(0.4)\cong\ar[d]^m &
\Hom(H_2(G,\bZ),k^\times) \ar[r]\ar[d]^m & 0\\
0\ar[r] & \Ext^1(H_1(G,\bZ),k^\times)\ar[r]&H^2(G,k^\times)\ar[r]^(0.4)\cong
&\Hom(H_2(G,\bZ),k^\times)\ar[r]&0} \]
We have $\Ext^1(H_1(G,\bZ),k^\times)=0$ since $k^\times$ is divisible, 
so $H^2(G,k^\times)\to \Hom(H_2(G,\bZ),k^\times)$ is an isomorphism. 
Since $\alpha$ has order $m$, its image in $\Hom(H_2(G,\bZ),k^\times)$ 
lifts to a \emph{surjective} element of order $m$
in $\Hom(H_2(G,\bZ),\mu_m)$. 
An inverse image $\tilde\alpha\in H^2(G,\mu_m)$ again has order $m$. Let 
\[ 1 \to Z \to \tilde G \to G \to 1 \]
be a central extension corresponding to $\tilde\alpha$, with $Z\cong\mu_m$.

Now choose a presentation of $G$ by generators and relations
\[ 1 \to R \to F \to G \to 1. \]
By freeness, the identity map on $G$ lifts to a map $F \to \tilde G$.
This map sends $R$ into $Z$, and $[F,F]$ into $[\tilde G,\tilde G]$.
It has $[F,R]$ in its kernel since $\mu_m$ is central. This gives us a map
\[ H_2(G,\bZ) = (R\cap [F,F])/[F,R] \to Z \]
which is the image of $\alpha$ in $\Hom(H_2(G,\bZ),\mu_m)$.
This map is surjective, but it lands in $Z\cap [\tilde G,\tilde G]$
and hence $Z\subseteq [\tilde G,\tilde G]$.

The formula for the idempotent $e$ can be found in the statement and 
proof of Th\'evenaz \cite[Chapter 2, Proposition 10.5] {Thevenaz:1995a}. 
The linear character $\chi$ sends each element $z\in Z$ to its image 
under $Z\cong\mu_m\hookrightarrow k^\times$.
\end{proof}

Let $P$ be an abelian $p$-group,  $L$ an abelian $p'$-subgroup  of   
$\Aut(P)$  and $\alpha \in  H^2( L, k^{\times} ) $.  Let $Z$, $\tilde G$ 
and $ \chi $ be as in the conclusion of Lemma~\ref{le:twistcentral}  
applied to $G= P \rtimes  L $  and  with $\alpha $ regarded as an 
element of $H^2(G, k^{\times}) $ via the pull back along $G \to G/P $.  
Let $ H$ be the full inverse image of $L$ in $\tilde G$.   Then 
$ \tilde G =  P \rtimes H $ and $ e$  is a central idempotent of $H$.

We have a natural  homomorphism 
\[ \rho\colon H \to \Hom(H,k^\times) \]
sending $g$ to $\rho(g)\colon h \to \chi[g,h]$. The kernel of
this map is $Z  (H)  $ and   the image is     $\Hom(H/Z(H), k^\times)$. 
We denote by 
\[ \bar\rho\colon H/Z(H)  \to \Hom(H/Z, k^\times) \]
the induced  isomorphism. 

Now  $P/\Phi(P)$ is naturally a faithful $\bF_pL$-module.
The extension of scalars $k\otimes_{\bF_p}P/\Phi(P)$
gives a $kL$-module isomorphic to $J(kP)/J^2(kP)$.
Let $\psi$ be the character of $kL$ on this module, and write
\[ \psi=\bigoplus_{i=1}^r \psi_i \]
where $r$ is the rank of $P/\Phi(P)$ and the $\psi_i$ are 
one dimensional $kL$-modules (there may be repetitions).
We choose an $H$-invariant complement $W$ for $J^2(kP)$
in $J(kP)$, and a basis $w_i$ of $W$ so that for $g \in H$
we have
\begin{equation}\label{eq:gwg^-1} 
gw_ig^{-1}=\psi_i(g)w_i. 
\end{equation}
Since a $p'$-group of automorphisms of an abelian $p$-group
preserves some decomposition into homocyclic summands (see for example 
Chapter~5, Theorem~2.2 in Gorenstein~\cite{Gorenstein:1968a}),
we may assume that 
\begin{align*}
kP&=k[w_1,\dots,w_r]/(w_1^{p^{n_1}},\dots,w_r^{p^{n_r}}) 
\end{align*}
with $n_1\ge \cdots \ge n_r$ and $|P|=p^{n_1+\dots+n_r}$.
Thus we have relations
\begin{equation}\label{eq:wipni}
w_i^{p^{n_i}}=0. 
\end{equation}

Applying the results of Section~\ref{se:class2},
the irreducible characters $\tau_\phi$ of $H$ lying over $\chi$
are in one to one correspondence with the one dimensional
characters $\phi$ of $Z(H)$ lying over $\chi$.
The corresponding central idempotents are
\begin{equation}\label{eq:ephi} 
e_\phi =\frac{1}{|Z(H)|} \sum_{h\in Z(H)}\phi(h^{-1})h. 
\end{equation}

Choose one of these, say $\tau=\tau_{\phi_0}$, and choose a matrix 
representation $T_{\phi_0}\colon H \to \Mat_m(k)$ affording $\tau_{\phi_0}$.
Then for each $\phi$ choose a one dimensional representation $\xi_\phi$
of $H$ whose restriction to $Z(H)$ is $\phi\phi_0^{-1}$ (and hence
whose restriction to $Z$ is trivial) and chosen so that $\xi_{\phi_0}=1$.
We assume that these $\xi_\phi$ have been chosen, one for each $\phi$,
and we define $T_\phi\colon H \to \Mat_m(k)$ via $T_\phi(h)=
\xi_\phi(h).T_{\phi_0}(h)$.
Then $T_\phi$ is a matrix representation affording $\tau_\phi$. So the map
\begin{align*} 
kHe &\to \Mat_m(k) \times \dots \times \Mat_m(k) & (|Z(H):Z|\text{ copies}) \\
he & \mapsto (T_{\phi_0}(h),\dots, T_\phi(h),\dots)
\end{align*}
is an isomorphism. Elements of $kHe$ of the form 
$\sum_\phi \xi_\phi^{-1}(h)e_\phi.h$ are sent to diagonal elements 
$(T_{\phi_0}(h),\dots,T_{\phi_0}(h))$, and therefore
span a copy of $\Mat_m(k)$ in $kHe$ containing $e$ as its identity 
element. Let us write $\fM$ for this subalgebra of $e.kH$.

Now for each $\psi_i$ and each $\phi$, the character 
$\phi.(\psi_i|_{Z(H)})$ is some $\phi'$, which we denote $\phi\psi_i$ 
for convenience. So $\xi_{\phi\psi_i}\xi_\phi^{-1}\psi_i^{-1}$ is 
trivial on $Z(H)$. Thus there exists an element $g_{i,\phi}\in H$
such that 
\begin{equation}\label{eq:xi-rho} 
\rho(g_{i,\phi})=\xi_{\phi\psi_i}\xi_\phi^{-1}\psi_i^{-1}, 
\end{equation}
where $\rho(g_{i,\phi})(h)=\chi([g_{i,\phi},h])$. We choose such 
elements $g_{i,\phi}$,  one for each $\psi_i$ and $\phi$.

On the other hand, using~\eqref{eq:gwg^-1} and~\eqref{eq:ephi},  we have
\[ \sum_{h\in Z(H)}\phi(h^{-1})w_ih = 
\sum_{h\in Z(H)}\phi(h^{-1})\psi_i(h^{-1})hw_i   \]
and so 
\begin{equation}\label{eq:we}
w_i e_\phi = e_{\phi\psi_i}w_i.
\end{equation}

\begin{lemma}
For $h\in H$ we have
\[ (g_{i,\phi}w_i)(\xi_\phi(h)^{-1}e_\phi .h) = 
(\xi_{\phi\psi_i}(h)^{-1}e_{\phi\psi_i}.h)(g_{i,\phi}w_i). \]
Thus $g_{i,\phi}w_ie_\phi=e_{\phi\psi_i}g_{i,\phi}w_i$ commutes with $\fM$.
\end{lemma}

\begin{proof}
Scalars commute with everything, and the $e_\phi$, being in $Z(kH)$, 
commute with all but the $w_i$.
By~\eqref{eq:gwg^-1} we have $hw_i=\psi_i(h)w_ix$, and by~\eqref{eq:we}
we have $w_ie_\phi=e_{\phi\psi_i}w_i$. 
We are in $kGe$, and $eg_{i,\phi}h=e\chi([g_{i,\phi},h])^{-1}hg_{i,\phi}$.
Putting these together gives
\[ \xi_\phi(x)^{-1}(g_{i,\phi}w_i)(e_\phi.h) = 
\xi_{\phi}(h)^{-1}\psi_i(h)^{-1}\chi([g_{i,\phi},h])^{-1}
(e_{\phi\psi_i}h)(g_{i,\phi}w_i). \]
Finally, applying~\eqref{eq:xi-rho}, the scalar on the right hand side is equal
to $\xi_{\phi\psi_i}(h)^{-1}$.

For the final statement, we have
\begin{align*}
(g_{i,\phi}w_ie_\phi)\left(\sum_{\phi'}\xi_{\phi'}^{-1}(h)e_{\phi'}.h\right) 
&=(g_{i,\phi}w_i)(\xi_\phi^{-1}(h)e_\phi.h) \\
&=(\xi_{\phi\psi_i}^{-1}(h)e_{\phi\psi_i}.h)(g_{i,\phi}w_i) \\
&=\left(\sum_{\phi'}
\xi_{\phi'}^{-1}(h)e_{\phi'}.h\right)(e_{\phi\psi_i}g_{i,\phi}w_i).
\qedhere
\end{align*}
\end{proof}

\begin{defn}
Let $\fA$ be the subalgebra of $kHe$ generated by the elements
$e_\phi$ and $g_{i,\phi}w_ie_\phi$. Thus by the lemma, $\fA$ and
$\fM$ commute.
\end{defn}

We claim that $\fA$ is a basic algebra of dimension 
$|P|\cdot |H : Z(H)|$, and that multiplication in $kHe$ induces an 
isomorphism
\[ \fA \otimes_k \fM \to kHe, \]
so that $kHe \cong \Mat_m(\fA)$. For this purpose, we shall
use the following.

\begin{lemma}\label{le:Bass}
Let $A\le B$ be $k$-algebras with $A$ an Azumaya algebra
(i.e., a finite dimensional central separable $k$-algebra). Then the
map $A \otimes_k C_B(A) \to B$ is an isomorphism.
\end{lemma}

\begin{proof}
See for example Chapter 3, Corollary 4.3, in Bass~\cite{Bass:1967a}.
\end{proof}

\begin{rk}
Note that the hypotheses of the lemma include the statement that
the identity element of $A$ is equal to the identity element of $B$.
\end{rk}

We display
$\fA$ as $kQ/I$ where $Q$ is a quiver and $I\le J^2(kQ)$ is an ideal 
of relations. The quiver $Q$ has $|Z(H):Z|$ vertices labelled $[\phi]$ 
corresponding to the idempotents $e_\phi\in kZ(H)$ lying over $\chi$, 
and directed edges labelled with the $w_i$ corresponding to 
\[g_{i,\phi}w_ie_\phi=e_{\phi\psi_i}g_{i,\phi}w_i=
e_{\phi\psi_i}g_{i,\phi}w_ie_\phi. \]
going from $[\phi]$ to $[\phi\psi_i]$.
For brevity, we can illustrate these vertices and directed edges as
\[ [\phi] \xrightarrow{\quad i\quad} [\phi\psi_i] \]

\begin{lemma}\label{le:quad-rel}
\begin{enumerate}

\item
For suitable elements $z_{i,j,\phi}\in Z(H)$, we have
\[ g_{j,\phi\psi_i}\,g_{i,\phi}=
g_{i,\phi\psi_j}\,g_{j,\phi}\,z_{i,j,\phi}. \]

\item
The following relations
hold in $\fA$:
\[ (g_{j,\phi\psi_i}w_je_{\phi\psi_i})(g_{i,\phi}w_ie_\phi) =
q_{i,j,\phi}(g_{i,\phi\psi_j}w_ie_{\phi\psi_j})(g_{j,\phi}w_ie_\phi) \]
where $q_{i,j,\phi}=\phi(z_{i,j,\phi})\in k^\times$.

\item
By changing the choices of $g_{i,\phi}$ by elements of $Z(H)$,
we may ensure that $z_{i,j,\phi}\in Z$ and $q_{i,j,\phi}=
\chi(z_{i,j,\phi})$.

\end{enumerate}
\end{lemma}

\begin{proof}
(i) By \eqref{eq:xi-rho}, we have
\begin{align*}
\rho(g_{j,\phi\psi_i}g_{i,\phi})&=\rho(g_{j,\phi\phi_i})\rho(g_{i,\phi}) \\
&= (\xi_{\phi\psi_i\psi_j}\xi_{\phi\psi_i}^{-1}\psi_j^{-1})(\xi_{\phi\psi_i}\xi_{\phi}^{-1}\psi_i^{-1}) \\
&= \xi_{\phi\psi_i\psi_j}\xi_{\phi}^{-1}\psi_j^{-1}\psi_i^{-1}.
\end{align*}
This is symmetric in $i$ and $j$, and so
\[ \rho(g_{j,\phi\psi_i}g_{i,\phi})=\rho(g_{i,\phi\psi_j}g_{j,\phi}). \]
Since the kernel of $\rho$ is $Z(H)$ it follows that for some element
$z_{i,j,\phi}\in Z(H)$ we have 
\[ g_{j,\phi\psi_i}\,g_{i,\phi}=g_{i,\phi\psi_j}\,g_{j,\phi}\,z_{i,j,\phi}. \]

(ii) This follows from the fact that we have 
$z_{i,j,\phi}\,e_\phi = \phi(z_{i,j,\phi})e_\phi$.

(iii) We apply Lemma \ref{le:coh} with $A =\mathrm{Irr}(H/Z(H))$, 
$B=\mathrm{Irr}(H/Z)$, $C = \mathrm{Irr}(Z(H)/Z)$, the map from $B$ to 
$C$  the  restriction map, $ D=  H/Z $, $ \pi $ the composition of  
the natural surjection  $ H/Z  \to  H/Z(H)$    with   $\bar  \rho$  and  
$ u_{\alpha} =   \xi_{ \alpha \phi_0} $,  
$\alpha \in \mathrm{Irr}  (Z(H)/Z)  $.    Let   
$ f_{\alpha, \beta }   \in    H/Z $  be as in  the conclusion of the 
lemma, and let $\tilde f_{\alpha, \beta}  \in H$ be any    lift  of 
$ f_{\alpha, \beta} $  to $H$.    
Denote   also   by $\psi_i$  the restriction of $\psi_i $ to  $ Z(H) $.  
So $ \psi_i^{-1} u_{\psi _i}  $ is an element of 
$  \mathrm{Irr}  (H/Z(H)) $. Choose an element $ g_{i} \in H $ such that 
$\rho (g_i) =   \psi_i^{-1} u_{\psi _i}  $   and   set  
\[g_{i, \phi} =    g_i  \tilde  f_{ \psi _i,   \phi  \phi_0^{-1}}  .\]
Then \begin{align*}\rho (g_{i, \phi}) & =  \psi_i^{-1} u_{\psi _i}  
\rho (\tilde f_{ \psi _i,   \phi  \phi_0^{-1}})\\
&=  \psi_i^{-1} u_{\psi _i} u_{\psi_i}^{-1} u_{ \phi \phi_0^{-1} }^{-1}  
u_{  \psi_i\phi\phi_0^{-1}}\\
&= \psi^{-1} \xi_{\phi}^{-1} \xi_{\psi_i\phi} 
 \end{align*}  and 
\begin{align*} g_{j,\phi\psi_i}\,g_{i,\phi}  Z 
&=  g_j  \tilde  f_{ \psi_j,   \phi  \psi_i\phi_0^{-1}} \,  
g_i\tilde  f_{ \psi _i,   \phi  \phi_0^{-1}} Z\\
& = g_j g_i  \tilde  f_{ \psi_j,   \phi  \psi_i \phi_0^{-1}}
\tilde  f_{ \psi _i,   \phi  \phi_0^{-1}}Z\\
&=  g_j g_i  \tilde  f_{ \psi_j,   \phi  \psi_i \phi_0^{-1}}
\tilde  f_{ \phi  \phi_0^{-1},  \psi _i}Z   \\
&= g_j g_i\tilde f_{\psi_j,  \phi\phi_0^{-1}}
\tilde f_{\psi_j\phi\phi_0^{-1},   \psi_i}\\
&= g_j g_ji\tilde f_{\psi_j,  \phi\phi_0^{-1}}
\tilde f_{ \psi_i, \psi_j\phi\phi_0^{-1}}Z\\
&=g_{i,\phi\psi_j}\,g_{j,\phi}Z.   \qedhere \end{align*}
 \end{proof}

\begin{lemma}\label{le:power-rel}
For each $i$ and each $\phi$ we have
\[ (g_{i,\phi\psi_i^{p^{n_i}-1}}w_ie_{\phi\psi_i^{p^{n_i}-1}})\dots
(g_{i,\phi\psi_i^2}\,w_ie_{\phi\psi_i^2})
(g_{i,\phi\psi_i}\,w_ie_{\phi\psi_i})(g_{i,\phi}\,w_ie_\phi)=0. \]
Here, we have written down the only composable sequence of arrows in 
$Q$ beginning with $e_\phi$ with each arrow involving $w_i$, and there 
are $p^{n_i}$ terms in the product:
\[ [\phi] \xrightarrow{\quad i \quad} [\phi\psi_i] \xrightarrow{\quad i \quad}
[\phi \psi_i^2]\xrightarrow{\quad i \quad} \cdots 
\xrightarrow{\quad i \quad} [\phi\psi_i^{p^{n_i}}]. \]
\end{lemma}

\begin{proof}
Relations~\eqref{eq:gwg^-1} and~\eqref{eq:we} allow us to push the $w_i$ 
terms past the other terms so that they are directly multiplied 
together. Then we can then use relation~\eqref{eq:wipni} to conclude 
that we get zero.
\end{proof}

\begin{theorem}\label{th:A}
The relations on the quiver algebra $kQ$ given by
\[ (g_{j,\phi\psi_i}w_je_{\phi\psi_i})(g_{i,\phi}w_ie_\phi) =
q_{i,j,\phi}(g_{i,\phi\psi_j}w_ie_{\phi\psi_j})(g_{j,\phi}w_ie_\phi) \]
and
\[ (g_{i,\phi\psi_i^{p^{n_i}-1}}w_ie_{\phi\psi_i^{p^{n_i}-1}})\dots
(g_{i,\phi\psi_i^2}\,w_ie_{\phi\psi_i^2})
(g_{i,\phi\psi_i}\,w_ie_{\phi\psi_i})(g_{i,\phi}\,w_ie_\phi)=0 \]
as in Lemmas~\ref{le:quad-rel} and~\ref{le:power-rel}
are a complete set of relations among the arrows 
$g_{i,\phi}w_ie_\phi$ of $Q$ to define the quotient $\fA$. 
Thus $\fA\cong kQ/I$ where $I\le J^2(kQ)$ is
the two-sided ideal generated by these relations.
\end{theorem}

\begin{proof}
Using Lemmas~\ref{le:quad-rel} and~\ref{le:power-rel}, we have an 
obvious homomorphism from the algebra $kQ/I$ given by these generators 
and relations to $\fA$. Now $kQ/I$ is a finite dimensional
algebra whose socle elements are products which involve $p^{n_i}-1$ 
arrows of type $i$ for each $i$. Such an element maps to something of 
the form (element of $H$)($w_1^{p^{n_1}-1}\dots w_r^{p^{n_r}-1}e_\phi$) 
in $\fA$, and such an element is non-zero in $kG$.
\end{proof}

\begin{theorem}
The multiplication in $kGe$ induces an isomorphism $\fA\otimes_k\fM\to 
kGe$.
\end{theorem}

\begin{proof}
Applying Lemma~\ref{le:Bass} with $A=\fM$ and $B$ the subalgebra 
generated by $\fA$ and $\fM$, we see that the given map is injective.
The dimensions are given by $\dim(\fA)=|Z(H):Z|\cdot |P|$, 
$\dim(\fM)=|H:Z(H)|$ and $\dim(kGe)=|G:Z|$, so $\dim(kGe)=
\dim(\fA)\cdot \dim(\fM)$ and the map is an isomorphism.
\end{proof}

\begin{cor}
The algebra $kGe$ is isomorphic to $\Mat_m(\fA)$. In particular,
$\fA$ is the basic algebra of $kGe$, and is Morita equivalent to it.\qed
\end{cor}

\bibliographystyle{amsplain}
\bibliography{../repcoh}

\newcommand{\noopsort}[1]{}
\providecommand{\bysame}{\leavevmode\hbox to3em{\hrulefill}\thinspace}
\providecommand{\MR}{\relax\ifhmode\unskip\space\fi MR }
\providecommand{\MRhref}[2]{%
  \href{http://www.ams.org/mathscinet-getitem?mr=#1}{#2}
}
\providecommand{\href}[2]{#2}
\begin{thebibliography}{1}

\bibitem{Bass:1967a}
H.~Bass, \emph{{Lectures on topics in algebraic $K$-theory}}, Tata Institute,
  Bombay, 1967.

\bibitem{Benson/Green:2004a}
D.~J. Benson and E.~Green, \emph{{Nonprincipal blocks with one simple module}},
  Quarterly J.\ Math (Oxford) \textbf{55} (2004), 1--11.

\bibitem{Benson/Kessar:2007a}
D.~J. Benson and R.~Kessar, \emph{{Blocks inequivalent to their Frobenius
  twists}}, J.~Algebra \textbf{315} (2007), 588--599.

\bibitem{EatonLivesey:2018a} C.~W. Eaton and M.~Livesey, 
\emph{{Donovan’s conjecture and blocks with abelian defect groups,}}
arXiv:1803.03539v1 (2018). 

\bibitem{EatonLivesey:2018b} C.~W. Eaton and M.~Livesey, 
\emph{{Towards Donovan's conjecture for abelian defect groups}},
arXiv:1711.05357v2 (2018).

\bibitem{Gorenstein:1968a}
D.~Gorenstein, \emph{{Finite groups}}, Harper \& Row, 1968.

\bibitem{Holloway/Kessar:2005a}
M.~Holloway and R.~Kessar, \emph{{Quantum complete rings and blocks with one
  simple module}}, Quarterly J.\ Math (Oxford) \textbf{56} (2005), 209--221.

\bibitem{Kessar:2004} 
R.~Kessar, \emph{{A remark on Donovan's
conjecture}}, Archiv Math. (Basel) {\textbf 82} (2004),
391--394.


\bibitem{Kulshammer:1985a}
B.~K\"ulshammer, \emph{{Crossed products and blocks with normal defect
  groups}}, Commun.\ in Algebra \textbf{13} (1985), no.~1, 147--168.

\bibitem{Thevenaz:1995a}
J.~Th\'evenaz, \emph{{$G$-algebras and modular representation theory}}, Oxford
  University Press, 1995.

\end{thebibliography}

\end{document}